\documentclass{amsproc}
\usepackage{graphicx}
\usepackage{amssymb}
\usepackage{epstopdf}
\usepackage{amsmath,amsthm,amscd,amssymb}
\usepackage{latexsym}
\usepackage{fancyhdr}
\usepackage{setspace}
\usepackage[colorlinks,citecolor=red,pagebackref,hypertexnames=false]{hyperref}
\usepackage{geometry}                
\numberwithin{equation}{section}

\theoremstyle{plain}
\newtheorem{theorem}{Theorem}[section]
\newtheorem{lemma}[theorem]{Lemma}
\newtheorem{corollary}[theorem]{Corollary}
\newtheorem{proposition}[theorem]{Proposition}

\theoremstyle{definition}
\newtheorem{definition}[theorem]{Definition}

\newtheorem{case[theorem]}{Case}

\theoremstyle{remark}

\numberwithin{equation}{section}

\begin{document}

\title{Wavelet decomposition and bandwidth of functions defined on vector spaces over finite fields} 

\author{Alex Iosevich, Allen Liu, Azita Mayeli and Jonathan Pakianathan}

\date{today}

\address{Department of Mathematics, University of Rochester, Rochester, NY}
\email{iosevich@math.rochester.edu}
\address{Penfield High School} 
\email{cliu568@gmail.com}
\address{Department of Mathematics, CUNY Queensborough, NY} 
\email{AMayeli@qcc.cuny.edu}
\address{Department of Mathematics, University of Rochester, Rochester, NY}
\email{jonpak@math.rochester.edu}

\thanks{The work of the first and last listed authors was partially supported by NSA Grant H98230-15-1-0319}

\maketitle

\begin{abstract}  In this paper we study how zeros of the Fourier transform of a function $f: \mathbb{Z}_p^d \to \mathbb{C}$  are related to the structure of the function itself. 
In particular, we introduce a notion of bandwidth of such functions and discuss its connection with the decomposition of this function into wavelets. 
Connections of these concepts with the tomography principle and the Nyquist-Shannon sampling theorem are explored.

We examine a variety of cases such as when the Fourier transform of the characteristic function of a set $E$ vanishes on specific sets of points,  affine subspaces, and algebraic curves.  In each of these cases, we prove properties such as equidistribution of $E$ across various surfaces and bounds on the size of $E$.  

We also establish a finite field Heisenberg uncertainty principle for sets that relates their bandwidth dimension and spatial dimension.
\end{abstract} 

\tableofcontents

\section{Introduction}

\vskip.125in

Let $E \subset {\Bbb Z}_p^d$, where $\mathbb{Z}_p$ is the prime field of size $p$ and ${\Bbb Z}_p^d$ is the $d$-dimensional vector space over ${\Bbb Z}_p$. Throughout the paper we shall identify $E \subset {\Bbb Z}_p^d$ with its indicator function $E(x)$ and the size of $E$ shall be denoted by $|E|$. More generally, consider $f: {\Bbb Z}_p^d \to {\mathbb C}$ and define its Fourier transform by the relation 
\begin{equation} \label{ftdef} \widehat{f}(m)=p^{-d} \sum_{x \in {\Bbb Z}_p^d} \chi(-x \cdot m) f(x), \end{equation} where 
$\chi(u)=e^{\frac{2 \pi i u}{p}}$, $u\in \Bbb Z_p$. 

The question we ask is, what can we say about the structure of $f: {\Bbb Z}_p^d \to {\mathbb C}$ given that $\widehat{f}$ is supported in a prescribed subset of ${\Bbb Z}_p^d$. In Euclidean space questions of this type have been studied for a long time in a variety of settings. For example, if $\mu$ is a compactly supported Borel measure on ${\Bbb R}^d$ and $\widehat{\mu}$ vanishes on an open ball, then $\mu$ is identically $0$ since it is not difficult to see that $\widehat{\mu}=\int e^{-2 \pi ix \cdot \xi} d\mu(x)$ is real analytic. Another example is provided by the Nyquist-Shannon sampling theorem. It says that if $f \in L^2(\Bbb R)$ and the Fourier transform  of $f$, given by $\widehat{f}(\xi)=\int_{\Bbb R} e^{-2 \pi i x \cdot \xi} f(x) dx$, $\xi \in {\Bbb R}$, vanishes outside the interval $[-B,B]$, then $f$ can be recovered by sampling on discrete points $f(kT), k\in \Bbb Z$, spaced by $T=\frac{1}{2B}$. See \cite{N28} and \cite{S49}. For a more modern treatment, see, for example, \cite{U00}. 

In the setting of vector spaces over ${\Bbb Z}_p$, motivated by the fact that aside from the origin, the zero set and support set of the Fourier transform of a rational valued function is a union of punctured lines (see \cite{HIPRV15}, a punctured line is a line through the origin with the origin taken out), we 
define the bandwidth of a complex valued function as the minimum number of lines which contain the support of its Fourier transform. We then show that this quantity determines the minimum number of wavelets, which are defined as linear combinations of indicator functions of a family of parallel hyperplanes, that this function decomposes into. These should be viewed as finite field analogs of the Euclidean wavelets, first introduced by Grossmann and Morlet in \cite{GM84}. The Euclidean wavelets involve dilations and translations of a fixed function. Our ${\Bbb Z}_p^d$ analog is a wavelet where the dilation structure is implicit in the dilation invariance of a subspace and translation manifests itself in that the planes are parallel. To put it another way, the wavelet analysis in 
${\Bbb Z}_p$ involves only one scale. In the last section of the paper we introduce wavelets over ${\Bbb Z}_{p^l}$ where the multi-resolution aspect of wavelet theory is present in the form of $l$ different scales. A more detailed analysis shall be carried out in a subsequent paper. 

Wavelets form an interesting basis to use when studying the vector space of functions $f: \mathbb{Z}_p^d \to \mathbb{C}$ and are discussed in Section~\ref{sec: wavelet}. In this section the decomposition of any function $f$ into wavelets is discussed as well as the corresponding tomography principle. It turns out that any function $f: \mathbb{Z}_p^d \to \mathbb{C}$ can be reconstructed purely from the knowledge of its masses on affine hyperplanes. See a seminal papers \cite{D88} and \cite{M89} for a treatment of compactly supported wavelets and related issues in Euclidean space. 

We use these results to understand direct connections between the bandwidth of the Fourier transform of the indicator function of a set and structural properties of the set itself in many instances. We also obtain a relationship between the bandwidth dimension (defined in later sections) of the indicator function of a set $E$ and the formal dimension of $E$ that is the finite field analogue of the Heisenberg uncertainty principle.

In a final section of the paper, we study another basis for the vector space of functions $f: \mathbb{Z}_p^d \to \mathbb{C}$ given by eigenfunctions of the Fourier transform. We prove among other things, that such an eigenfunction can be the characteristic function of a set $E$, only when $d$ is even and $E$ is a Lagrangian subspace of $\mathbb{Z}_p^d$. 

Throughout this paper, we work exclusively over prime fields ${\Bbb Z}_p$ where $p$ is a prime. There is no loss in generality in doing this rather than working in general finite fields as we will be studying properties of the Fourier transform which depend solely on the additive structure of the vector spaces we work in. The vector space $\mathbb{F}_q^d$ over a finite field $\mathbb{F}_q$ where $q=p^{\ell}, p$ a prime, will be additively isomorphic to the vector space  
$\mathbb{Z}_p^{d\ell}$ over the prime field $\mathbb{Z}_p$.

\vskip.25in 

\section{Basic Properties}
\vskip.25in 

Let us begin with a very simple case when $\widehat{f}(m)$ is supported at the origin $(0, \dots, 0)$, which shall henceforth be denoted by 
$\vec{0}$. Recall that with $\widehat{f}$ defined as in (\ref{ftdef}), 
$$ f(x)=\sum_{m \in {\Bbb Z}_p^d} \chi(x \cdot m) \widehat{f}(m)=\widehat{f}(\vec{0}),$$ which shows that if $\widehat{f}$ is supported at a single point, then $f$ is constant. Remarkably we can recover the same conclusion with a much weaker hypothesis on $\widehat{f}$ provided that the image of $f$ is contained in ${\mathbb Q}$, the field of rational numbers. Before stating our first result, we need a bit of notation. Here and throughout, if $g: {\Bbb Z}_p^d \to {\mathbb C}$, then 
$$Z(g)=\{x \in {\Bbb Z}_p^d: g(x)=0\}.$$ 

We also need the following notion. 

\begin{definition} We say that $A \subset {\Bbb Z}_p^d$ is a {\it compass set} if given any $x \in {\Bbb Z}_p^d$, there exists $y \in A$ such that $x=ty$ for some $t \in {\Bbb Z}_p$. \end{definition} 

\begin{theorem} \label{nobigzeroes} Let $f: {\Bbb Z}_p^d \to {\mathbb Q}$ and suppose that $Z(\widehat{f})$ is a compass set. Then $f$ is constant. \end{theorem}

This result is a consequence of the following general theorem discussed in \cite{HIPRV15}. 
\begin{theorem}[Characterization of Fourier transform of rational-valued functions]
\label{thm: rationalFourier}
 Fix $p$ a prime and let $\{ g_r \}_{r \in \mathbb{F}_p^*}$ denote 
$Gal(\mathbb{Q}(\xi)/\mathbb{Q})$.
Let $f: \mathbb{F}_p^d \to \mathbb{Q}$ be a rational-valued function. Let $m \in \mathbb{F}_p^d \backslash \{ 0 \}$ be a nonzero vector. Then for all 
$r \in \mathbb{F}_p^*$ we have:
$$ \hat{f}(rm)=g_r(\hat{f}(m)). $$

In particular $\hat{f}(m) = 0$ implies $\hat{f}(rm)=0$ for all $r \in \mathbb{F}_p^*$. Furthermore if we choose a set $M \subseteq \mathbb{F}_p^d$ such that $M$ contains exactly one nonzero element from each line through the origin 
(so $|M| = \frac{p^d-1}{p-1}$) and set 
$$Ave(f) = \frac{1}{p^d} \sum f(x)  \in \mathbb{Q}=\hat{f}(0)$$ to be the average of $f$ then the map
$$ \Phi: \mathbb{Q}[\mathbb{F}_p^d] \to \mathbb{Q} \times \mathbb{Q}(\xi)^{|M|} $$ given by
$$\Phi(f)=(\hat{f}(0), (\hat{f}(m))_{m \in M})$$ is a $\mathbb{Q}$-vector space isomorphism.
\end{theorem}

\vskip.125in 

Indeed, in view of Theorem \ref{thm: rationalFourier} and the assumption that $Z(\widehat{f})$ is a compass set, $\widehat{f}$ vanishes on all of ${\Bbb Z}_p^d$ except for the origin. It follows that 
$$ f(x)=\sum_{m \in {\Bbb Z}_p^d} \chi(x \cdot m) \widehat{f}(m)=\widehat{f}(\vec{0})=p^{-d} \sum_{y \in {\Bbb Z}_p^d} f(y)$$ for all $x \in {\Bbb Z}_p^d$. This completes the proof of Theorem \ref{nobigzeroes}. 

\vskip.125in 

There are two major principles contained in $\ref{thm: rationalFourier}$, one is the {\bf vanishing principle} for rational valued functions which says that if $\hat{f}(m)=0$ for some nonzero 
vector $m$, then $\hat{f}$ vanishes on the whole punctured line (line with origin taken out) through $m$.  The vanishing principle does not in general hold for complex-valued functions. On the other hand, we will soon see that $\ref{thm: rationalFourier}$ also contains the principle that the function can be recovered purely from knowledge of its masses on affine hyperplanes. This latter principle, the {\bf tomography principle}, holds for all complex-valued functions and is the basis for a wavelet decomposition that we will discuss soon.

Theorem \ref{thm: rationalFourier} allows us to introduce the notion of bandwidth in the context of ${\Bbb Z}_p^d$ in analogy with the Nyquist-Shannon theorem. 

\begin{definition} Suppose that $f: {\Bbb Z}_p^d \to {\mathbb C}$. We define the {\it coarse bandwidth} of $f$, denoted by $cbw(f)$, to be the number of lines $l$ through the origin such that $\widehat{f}$ does not vanish identically on the punctured line $l \backslash \vec{0}$. Equivalently, $cbw(f)$ is the minimum number of lines in a cone such that the support of $\hat{f}$ is contained in the cone. (Here we make the convention that the cone determined by an empty set of lines consists of just the origin.) Thus $0 \leq cbw(f) \leq \frac{p^d-1}{p-1}$. 

The {\it bandwidth} of $f$, denoted by $bw(f)$ is defined by $bw(f)=cbw(f) \frac{p-1}{p^d-1}$ and hence is a number in $[0,1]$. Note, the smaller the bandwidth, the smaller the 
support set of $\hat{f}$ is.

We define the {\it bandwidth dimension} of $f$ in a slightly tricky way motivated by a later result in Theorem~\ref{theorem: bandwidth codimension}. We define the bandwidth dimension of $f$, denoted by $bwd(f)$, to be the unique real number 
$bwd(f) \in [0,d]$ such that $cbw(f) = \frac{p^{bwd(f)}-1}{p-1}$. As the function $f: [0,d] \to [0,\frac{p^d-1}{p-1}]$ given by $f(b)=\frac{p^b-1}{p-1}$ is strictly monotonic, $bwd(f)$ exists in $[0,d]$ and is unique. Intuitively, $bwd(f)$ would be the dimension of the vector space over $\mathbb{Z}_p$ (if such existed) that had $cbw(f)$ many lines in it.
\end{definition} 

In view of Theorem \ref{thm: rationalFourier} we have the following result. 

\begin{corollary} Suppose that $f: {\Bbb Z}_p^d \to {\mathbb C}$. Then the (coarse) bandwidth (dimension) of $f$ is $0$ if and only if $f$ is a constant. \end{corollary} 

\vskip.125in 

\section{Wavelets and the Tomography Principle}
\label{sec: wavelet}

In this section, we study functions whose Fourier transform is supported on a line.

\begin{definition}[Wavelets] Let $s$ be a nonzero vector and let $H_{s,t} = \{ x \in \mathbb{Z}_p^d: x \cdot s = t \}$ for $0 \leq t \leq p-1$ be the corresponding family of affine hyperplanes 
perpendicular to $s$. A linear combination of the characteristic functions of these parallel hyperplanes:
$$
f=\sum_{t=0}^{p-1} c_t 1_{H_{s,t}}, \text{ where } c_t \in \mathbb{C}
$$
will be called a {\bf wavelet in the direction of $s$}. This wavelet will be called {\bf reduced} if $c_0=0$. The mass of a wavelet is 
$$m(f)=p^{d-1} \sum_{t=0}^{p-1} c_t.$$ A wavelet $f$ will be called {\bf massless} if $m(f) = 0$. It will be called a 
{\bf rational wavelet} if the $c_t \in \mathbb{Q}$. Finally it will be called a {\bf wavelet density} if $c_t \in \mathbb{R}, c_t \geq 0$ and it has mass $1$. 
\end{definition}

Note the picture of a wavelet is of a function that consists of a superposition of a sequence of parallel ``wavefronts'' (which are the parallel affine hyperplanes) where it has constant value (signifying a general amplitude) on each of the individual waves. Thus a wavelet density is a probability density which has uniform density 
on each individual wave in a set of parallel waves.

Notice that 
$$\sum_{t=0}^{p-1} 1_{H_{s,t}} = \mathfrak{1},$$ the constant function with value $1$. Therefore we may always write any wavelet in the form 
$$f=\sum_{t=0}^{p-1} c_t 1_{H_{s,t}}=\sum_{t=1}^{p-1} (c_t-c_0) 1_{H_{s,t}} + c_0,$$ a reduced wavelet plus a constant, or in the form
$$f=\sum_{t=0}^{p-1} (c_t - A) 1_{H_{s,t}} + A,$$ a massless wavelet plus a constant, where $A=\frac{m}{p^d}$ and $m=m(f)$ is the mass of $f$. 

\begin{theorem}[Wavelet Theorem]
\label{thm: wavelet} Let $f: \mathbb{Z}_p^d \to \mathbb{C}$ be a non constant function. Then the following are equivalent: 

\vskip.125in 

a) $cbw(f)=1$

b) $support(\hat{f}) = \{ m \in \mathbb{Z}_p^d: \hat{f}(m) \neq 0 \}$ is contained in a line. 

c) $f$ is a wavelet. 

\end{theorem}

\vskip.125in 

\begin{proof}
The first equivalence is immediate by definitions so it remains to show that $Z(\hat{f})$ is contained in a line if and only if $f$ is a wavelet. 

An easy computation shows that $\widehat{1_{H_{s,0}}}(m) = \frac{1}{p} 1_{L_s}$ where $L_s$ is the line through $s$ and the origin. $1_{H_{s,t}}$ is a translation of 
$1_{H_{s,0}}$ and a quick computation reveals that $\widehat{1_{H_{s,t}}}(m)=0$ when $m \notin L_s$ and 
$$ \widehat{1_{H_{s,t}}}(ks) = \frac{\chi(-kt)}{p}$$ for $0 \leq k \leq p-1.$

Thus all the $1_{H_{s,t}}$s have their Fourier transform supported in the line $L_s$ and hence by linearity so does any wavelet in the direction of $s$. 

We now prove the converse. Take any function $f$ whose Fourier transform $\hat{f}$ is supported on the line $L_s$ though the origin and $s$. 
Since the functions 
$$\{ \theta_t(ks)=\chi(-kt): 0 \leq t \leq p-1 \}$$ are a complex basis for the complex valued functions on $L_s$, we may write 
$$\hat{f}(ks) = \sum_{t=0}^{p-1} c_t \frac{ \chi(-kt)}{p}$$ for unique $c_t \in \mathbb{C}$. By the computations in the previous paragraph, it follows that 
$$\hat{f} =\sum_{t=0}^{p-1} c_t \widehat{1_{H_{s,t}} }$$ and so we see that $f$ is a wavelet upon using Fourier inversion.
\end{proof}

\vskip.125in 

\begin{lemma}{( Wavelet Lemma )}
\label{lem: wavelet}
Let $f: \mathbb{Z}_p^d \to \mathbb{C}$ and let $s$ be a nonzero vector in $\mathbb{Z}_p^d$. Then for all $0 \leq k \leq p-1$ we have
$$
\hat{f}(ks) = \frac{1}{p} \sum_{t=0}^{p-1} \chi(-kt) \frac{m_{s,t}(f)}{p^{d-1}} =  \sum_{t=0}^{p-1} \frac{m_{s,t}(f)}{p^{d-1}} \widehat{1_{H_{s,t}}}(ks)
$$
where $m_{s,t}(f)=\sum_{x \in H_{s,t}} f(x)$ is the mass of $f$ on the affine hyperplane $H_{s,t}$.
\end{lemma}
\begin{proof}
By definition $\hat{f}(ks) =\frac{1}{p^d} \sum_{x \in \mathbb{Z}_p^d} f(x) \chi(-k x \cdot s)$. Partitioning space into the $p$ parallel hyperplanes $H_{s,t}$ then yields:
$$
\hat{f}(ks) = \frac{1}{p^d} \sum_{t=0}^{p-1} \sum_{x \in H_{s,t}} f(x) \chi(-kt) = \sum_{t=0}^{p-1} \frac{m_{s,t}(f)}{p^{d-1}}\frac{\chi(-kt)}{p}
$$
which yields the lemma immediately when used together with the computation $\widehat{1_{H_{s,t}}}(ks)=\frac{\chi(-kt)}{p}$ done in the proof of the last theorem. 
\end{proof}

\begin{definition} Given $f: \mathbb{Z}_p^d \to \mathbb{C}$ and nonzero vector $s$, we denote by $f_s$, the wavelet associated to $f$ in the direction $s$, to be 
$$
\sum_{t=0}^{p-1} \frac{m_{s,t}(f)}{p^{d-1}} 1_{H_{s,t}}.
$$
Note that the mass of this wavelet is 
$$\sum_{t=0}^{p-1} \frac{m_{s,t}(f)}{p^{d-1}} p^{d-1} = \sum_{x \in \mathbb{Z}_p^d} f(x) = m(f).$$ 
\end{definition}

The wavelet $f_s$ just defined is the unique wavelet such that $\hat{f_s}=\hat{f}$ on $L_s$. 

\begin{theorem}( Wavelet Decomposition Theorem )
\label{thm: planes}
Let $f: \mathbb{Z}_p^d \to \mathbb{C}$ be a complex-valued function. Let $P_f$ denote the set of $cbw(f)$ lines through the origin that contain $support(\hat{f})$. We have the following explicit decomposition of $f$ into $cbw(f)$ many wavelets $f_s$, associated to $f$:
$$
f(x)= c + \sum_{\ell \in P_f} f_s(x) = c+ \sum_{l \in P_f} \left(\sum_{t=0}^{p-1}  \frac{m_{s,t}(f)}{p^{d-1}} 1_{H_{s,t}}(x)\right)
$$
where we choose a unique $s$ on each line $\ell$ in $P_f$ and $c=(1-cbw(f))\frac{m(f)}{p^d}$ is just a constant.

This sum can be rearranged into a sum of $cbw(f)$ reduced wavelets:
$$
f(x)= d+ \sum_{\ell \in P_f} \left(\sum_{t=1}^{p-1}  \frac{m_{s,t}(f)-m_{s,0}(f)}{p^{d-1}} 1_{H_{s,t}}(x)\right)
$$
where $d=c + \sum_{\ell \in P_f} \frac{m_{s,0}(f)}{p^{d-1}}$ is a new constant. 
Finally it can also be rearranged into a sum of $cbw(f)$ massless wavelets:
$$
f(x) = m(f) + \sum_{\ell \in P_f} \left(\sum_{t=0}^{p-1} \frac{pm_{s,t}(f)-m(f)}{p^d} 1_{H_{s,t}}(x)\right).
$$
\end{theorem}
\begin{proof}
Let $P_f$ be the set of (punctured) lines where $\hat{f}$ does not vanish identically. We can write $\hat{f} = \sum_{L \in P_f} \hat{f} |_{L} + c \delta$ where $\hat{f} |_{L}$ is the restriction of $\hat{f}$ to the line $L$ 
which is zero off the line $L$ and $\delta$ is the Kronecker delta function (which is needed to take care of the overlapping contributions of the $\hat{f} |_L$ at the origin).

By the wavelet lemma~\ref{lem: wavelet}, we have that $\hat{f} |_{L} = \hat{f_s}$ where $f_s$ is the wavelet associated to $f$ in the direction $s$ for any nonzero $s \in L$. 

The theorem then follows by taking inverse Fourier transform. The values of the constants $c$ and $d$ can be obtained by taking masses of both sides of the equation.
\end{proof}
Furthermore, it is not difficult to see from the results above that the decomposition into reduced/massless wavelets is unique.

The following corollaries are immediate consequences of Theorem~\ref{thm: planes}

\begin{corollary}
Let $f: \mathbb{Z}_p^d \to \mathbb{C}$ and $K$ be a subfield of $\mathbb{C}$. Then $f$ is $K$-valued if and only if all its masses $\{ m_{s,t}(f): s \in \mathbb{Z}_p^d - \{ 0 \}, 0 \leq t \leq p-1 \}$ are. 
\end{corollary}

\begin{corollary}
A rational valued function $f$ is the sum of $cbw(f)$ many rational wavelets and a constant. A probability density $f$ is the sum of $cbw(f)$ wavelet densities and a constant. 
\end{corollary}

\begin{corollary}(Tomography Principle) 
A function $f: \mathbb{Z}_p^d \to \mathbb{C}$ is uniquely determined by its masses 
$$\{ m_{s,t}(f): s \in \mathbb{Z}_p^d - \{ 0 \}, 0 \leq t \leq p-1 \}$$ on affine hyperplanes. 
\end{corollary}

\begin{corollary}(Wavelet Basis)
Fix $p$ a prime, $s \in \mathbb{Z}_p^d - \{ 0 \}$. Let $V_s$ be the set of wavelets supported in the $s$-direction then $\dim(V_s)=p$. 
Let $\bar{V}_s$ be the set of reduced wavelets supported in the $s$-direction then $\dim(\bar{V}_s)=p-1$. If $W$ is the $\mathbb{C}$-vector space of $\mathbb{C}$-valued functions on $\mathbb{Z}_p^d$ then $W= \mathbb{C} \oplus (\oplus_{s \in P} \bar{V}_s)$ where $P$ is a compass set containing one nonzero $s$ for every direction and the extra factor $\mathbb{C}$ corresponds to the constant functions.
\end{corollary}

\vskip.25in 

\section{Bandwidth and equidistribution theorems}
\vskip.125in

In general, the wavelet decompositions give us a reasonable way to explicitly classify and construct all rational-valued functions that have a given support.
\vskip.125in
A few simple consequences of Theorem \ref{thm: planes} follow.
\vskip.125in

\begin{definition} Given a $k$-dimensional subspace $V$ of ${\Bbb Z}_p^d$ we say that $V'$ is a parallel affine subspace if it is of the form $x+V$ for some $x \in {\Bbb Z}_p^d$. Note that there are $p^{d-k}$ affine subspaces parallel to a given $k$-dimensional subspace and they are all disjoint. \end{definition} 
 \vskip.125in
\begin{corollary}
\label{cor: perp}
Let $f: \mathbb{Z}_p^d \to \mathbb{C}$. If $\widehat{f}$ is supported on a subspace $V$, then $f$ is constant along each affine subspace parallel to $V^{\perp}$. 
\end{corollary}

\begin{proof}
Use the wavelet decomposition of $f$ given by Theorem~\ref{thm: planes}. As each line $\ell \in P_f$ lies in $V$, the hyperplanes in the wavelets in this decomposition all contain 
$V^{\perp}$ and thus each can be decomposed as a disjoint union of parallel subspaces to $V^{\perp}$. Thus $f$ is decomposed as a constant plus a linear combination of 
parallel spaces to $V^{\perp}$ and the corollary clearly follows. 
\end{proof}

We next use wavelets to give a lower bound on the bandwidth of the indicator function of a set unless the set is of a very special form.

\begin{theorem} 
\label{thm: lowerboundcbw1}
Let $E \subset {\Bbb Z}_p^d$. Then either $E$ is a union of parallel lines or $cbw(E)>d$. 
\end{theorem}

\begin{proof}
Assume for the sake of contradiction that $cbw(E) \leq d$.  
If the lines in the support of $\widehat{E}$ are not a basis of $\mathbb{Z}_p^d$ then $Supp(\hat{E}) \subseteq V$ for some proper subspace $V$ of $\mathbb{Z}_p^d$.
Thus the characteristic function of $E$ is constant on subspaces parallel to $V^{\perp}$ by Corollary~\ref{cor: perp}. Thus $E$ is a disjoint union of 
subspaces parallel to $V^{\perp}$. As $\dim(V^{\perp}) \geq 1$, $E$ is then a disjoint union of parallel lines.  

Otherwise, the lines in the support of $E$ are determined by a basis $p_1,p_2, \dots, p_d$ of $\mathbb{Z}_p^d$ and hence $cbw(E)=d$.   
Thus for every $1 \leq i \leq d, 0 \leq j \leq p-1$ there exists a unique point $q_{i}^j \in \mathbb{Z}_p^d$ such that 
$p_i \cdot q_{j}^i = j$ and $p_k \cdot q_{j}^{i} = 0$ for $k \neq i$.  Now, using the reduced wavelet decomposition given by Theorem~\ref{thm: planes}, 
$$
1_E(x) = d + \sum_{i=1}^d (\sum_{t=1}^{p-1} c_{i,t} 1_{H_{p_i,t}}(x)).
$$
Plugging in $x=q_{j}^i$ for any $1 \leq i \leq d$ and $j \neq 0$ yields $1_E(q_{j}^i) = d + c_{i,j} \in \{0,1\}$. While plugging in $x=q_{0}^i$ yields $1_E(q_{0}^i) = d \in \{0,1\}$. 
Thus $c_{i,j} \in \{ -1, 0, 1 \}$ for all $1 \leq i \leq d, 1 \leq j \leq p-1$. Using the formula for 
$$c_{i,j} = \frac{m_{p_i,j}(E) - m_{p_i,0}(E)}{p^{d-1}}$$ from Theorem~\ref{thm: planes} and 
the fact that $0 \leq m_{p_i, j} \leq p^{d-1}$ we see that this forces either the equality of all the $\{ m_{p_i,j}(E) \}_{j=0}^{p-1}$ or that $\{ m_{p_i,j}(E) \}_{j=0}^{p-1}=\{0, p^{d-1} \}$.
In the former case, $E$ equidistributes on the hyperplanes $\{H_{p_i,t} \}$ which would mean that $\hat{E}$ vanishes on the line through $p_i$ contrary to our assumptions. 
Thus we have the latter case which implies that $E$ is a disjoint union of parallel hyperplanes and hence also a disjoint union of parallel lines.
\end{proof}

Note, the bound in Theorem \ref{thm: lowerboundcbw1} is actually sharp in dimension $d=2$ for all values of a prime $p$.  For example, consider the set $E$ of points 
$$\{(x,y) \in \{0,1,\dots, p-1\}^2 : x+y\geq p \} \text{ mod } p.$$ Then $1_E$ can be expressed as the sum of three reduced wavelets 
$$\sum_{i=1}^{p-1} \left(\frac{i}{p}1_{x \equiv i}+\frac{i}{p}1_{y \equiv i}-\frac{i}{p}1_{x+y \equiv i}\right).$$

Thus $cbw(E)=3$ since $\hat{E}$ is supported on the three lines through $(0,1),(1,0)$ and $(1,1)$.

\vskip.125in 

The following general equidistribution theorem will be useful later in the paper.

\begin{theorem}(General Equidistribution)
\label{thm: equidist} Let $f: {\Bbb Z}_p^d \to \mathbb{C}$. Then $\widehat{f}$ vanishes on a punctured (origin taken out) $k$-dimensional subspace $V$ if and only if $f$ equi-distributes on affine subspaces parallel to $V^{\perp}$. More precisely, $m_{V'}(f)$ is a constant when $V'$ ranges over affine subspaces parallel to $V^{\perp}$. 

In particular, if $E$ is subset such that $\hat{E}$ vanishes on a $k$-dimensional subspace then $|E|$ is a multiple of $p^{k}$. Thus if $E$ is nonempty, $|E| \geq p^k$.
\end{theorem} 
\begin{proof} 
Set $W=V^{\perp}$ and note that $support(\hat{W})=V$. Then $\hat{f} \hat{W}$ is supported at the origin by assumption and so $\hat{f}\hat{W}=c\delta$ where $\delta$ is the Kronecker 
delta function and $c$ is some constant. Using Fourier inversion we find that $f \star W = c$ where $\star$ is discrete convolution. Thus for any $x$, 
$$
\sum_{y \in W} f(x+y) = c
$$
or in other words $m_{x+W}(f)$ is constant as $x+W$ ranges over subspaces parallel to $W=V^{\perp}$.

When $E$ is the indicator function of a set, this means $E$ has $c \geq 1$ elements in each of the $p^k$ parallel subspaces of $V^{\perp}$ and so 
$|E|=cp^k$ is a multiple of $p^k$ as desired.
\end{proof}

Now we prove a theorem about possible vanishing cones.

\begin{theorem} 
\label{thm: lines}
Suppose that we are working in ${\Bbb Z}_p^d$ and we have a set $S$ of punctured lines through the origin with $|S| < \frac{p^{d-k}-1}{p-1}$ (where $k$ is an integer with $0 \leq k < d$).  Then there exists a $(k+1)$-dimensional subspace of ${\Bbb Z}_p^d$ that does not intersect $S$.
\end{theorem}

$\textit{Proof:}$ We will use induction on $k$.

For the base case, $k=0$, it suffices to note that there are a total of $\frac{p^{d}-1}{p-1}$ lines through the origin and thus we can choose a line that is not part of $S$.

For the induction step, assume that we have proved the claim for $k=k_0-1$.  When $k=k_0 \geq 1$, we first find a $k_0$-dimensional subspace, $V$, that does not intersect $S$ (using the induction hypothesis).  Now there are 
$$\frac{p^d-p^{k_0}}{p^{k_0+1}-p^{k_0}}=\frac{p^{d-k_0}-1}{p-1}$$ different $k_0+1$-dimensional subspaces that contain $V$.  These $k_0+1$-dimensional subspaces must all be disjoint outside of $V$ since $V$ has dimension $k_0$. Using the fact that $|S| < \frac{p^{d-k_0}-1}{p-1}$, we see that there must be some $k_0+1$-dimensional subspace that does not intersect $S$.

\subsection{Finite field Heisenberg uncertainty principle}

Combining Theorem \ref{thm: equidist} with Theorem \ref{thm: lines}, we obtain the following result. 
\vskip.125in 
\begin{theorem}(Finite field Heisenberg Uncertainty Principle)
\label{theorem: bandwidth codimension}
Let $E$ be a nonempty set with coarse bandwidth $cbw(E)$ then $$((p-1)cbw(E)+1)|E| \geq p^d.$$
Let $\dim(E)=log_p(|E|)$ be the formal dimension of $E$. Let $c(E)=d-\dim(E)$ be the formal co-dimension of $E$. Then
$$ cbw(E) \geq \frac{p^c-1}{p-1} $$
where the right hand quantity is the formal number of lines in a fictionary vector space of dimension $c$. Thus $$bwd(E) \geq c(E),$$ or, equivalently, 
$$bwd(E)+\dim(E) \geq d,$$ which means that the bandwidth dimension is bounded below by the formal co-dimension.
\end{theorem}

\vskip.125in 

\begin{proof} We find the unique integer $k \geq 1$ such that $\frac{p^{k-1}-1}{p-1} \leq cbw(E) < \frac{p^k-1}{p-1}$. Now, applying  Theorem \ref{thm: lines}, there exists a   $d-k+1$-dimensional (punctured) subspace on which $\widehat{E}$ vanishes.  Combining this with Theorem~\ref{thm: equidist}, shows that $|E| \geq p^{d-k+1}$ and therefore $((p-1)cbw(E)+1)|E| \geq p^{k-1}p^{d-k+1}=p^d$.
\end{proof}

The last theorem, shows that a set $E$ of small (spatial) dimension must have large bandwidth dimension and vice versa. This is a direct analogue of the classical 
Heisenberg uncertainty principle.

\vskip.125in 

\vskip.125in 
\section{Fourier Transform on certain Algebraic Curves}
We now explore situations where $\widehat{f}$ vanishes on various algebraic varieties.

\begin{definition} Say a function $f$ is $\textit{good}$ if $\widehat{f}$ is supported on the set of points $ (x_1,x_2 \ldots ,x_d) $ satisfying $x_1^{2}+x_2^{2}+ \ldots + x_d^2=0$. \end{definition}

\begin{theorem} \label{paraboloid} Suppose that $f: {\Bbb Z}_p^d \to {\mathbb Q}$ and 
$$Z(\widehat{f}) \supset \{x \in {\Bbb Z}_p^d: x_d=x_1^2+x_2^2+\dots+x_{d-1}^2 \},$$ the paraboloid. Then if we let $f_a$ to be the function $f$ restricted to the plane $x_d=a$  (so $f_a$ is a function from $\mathbb{Z}_p^{d-1}$ to $\mathbb{Q}$) then for any $a,b$, $f_a-f_b$ is good as a function in $d-1$ dimensions. 

\end{theorem} 
\begin{proof} First, we claim that the set of directions determined by the paraboloid consists of all directions except those of the form $ (x_1,x_2 \ldots ,x_d) $ where $x_d \neq 0$ and $ x_1^2+x_2^2+\dots+x_{d-1}^2=0$ (call this type 1) or $x_d=0$ and $x_1^2+x_2^2+\dots+x_{d-1}^2 \neq 0$ (call this type 2).
Indeed, first if both $x_d$ and $ x_1^2+x_2^2+\dots+x_{d-1}^2$ are nonzero, the point 
$$\frac{x_d}{x_1^2+x_2^2+\dots+x_{d-1}^2}(x_1,x_2 \ldots ,x_d)$$ is on the paraboloid and if both $x_d$ and $ x_1^2+x_2^2+\dots+x_{d-1}^2$ are zero, the point $ (x_1,x_2 \ldots ,x_d) $ itself is on the paraboloid.

Now, we apply Theorem \ref{thm: planes} to complete the proof.  A given wavelet perpendicular to a direction of type 2 intersects all planes $x_d=a$ in the same way and hence contributes nothing to any of the differences $f_a-f_b$.  A given wavelet perpendicular to a direction of type 1 intersects the planes $x_d=a$ along $d-1$-dimensional wavelets perpendicular to directions of the form $(a_1,a_2, \cdots a_{d-1})$ for $ a_1^2+a_2^2+\dots+a_{d-1}^2=0$.  Each of these $d-1$-dimensional wavelets is $\textit{good}$ (viewed as a function over ${\Bbb Z}_p^{d-1}$) by Theorem \ref{thm: wavelet}.  Hence, any difference $f_a-f_b$ is a linear combination of $\textit{good}$ functions and hence is also $\textit{good}$.
\end{proof}

The situation becomes a bit more elaborate if the paraboloid is replaced by a sphere.
\\ 
In the following section, we present some properties of $\textit{good}$ functions.

\begin{theorem} Let $f: {\Bbb Z}_p^2 \to {\mathbb Q}$. Suppose that $\widehat{f}$ vanishes on $S_a \cup S_b$, where 
$$S_a=\{x \in {\Bbb Z}_p^d: x_1^2+x_2^2=a \}.$$ 

where $a$ is a quadratic residue modulo $p$ and $b$ is not. 

\vskip.125in 

i) Suppose that $p \equiv 3 \mod 4$. Then $f$ is constant. 

\vskip.125in 

ii) Suppose that $p \equiv 1 \mod 4$ and $f(x)=E(x)$ where $E \subset {\Bbb Z}_p^2$. Let 
$$L^{+}=\{(t,it): t \in {\Bbb Z}_p \} \ \text{and} \ L^{-}=\{(t,-it): t \in {\Bbb Z}_p \}.$$ 

Then $E$ is either a union of lines parallel to $L^{+}$ or a union of lines parallel to $L^{-}$. 

\end{theorem} 

\vskip.125in 
\begin{proof} This Theorem is a relatively direct consequence of Theorem \ref{thm: planes}.  We see that the support of $\widehat{f}$ must be contained in the set of points such that $x_1^2+x_2^2=0$.  For part (i), this set is empty so $f$ must be constant.  For part (ii) this set consists of $L^{+}$ and $L^{-}$ so we again obtain the desired statement.
\end{proof}

\begin{theorem} Let $f: {\Bbb Z}_p^d \to {\mathbb Q}$ where $d$ is even and $p>2$. Suppose that $\widehat{f}$ vanishes on $S_a \cup S_b$, where $S_a$ is defined as above (in higher dimensions), $a$ is a residue modulo $p$ and $b$ is not. Then $f$ is equi-distributed on spheres of non-zero radius centered at an arbitrary point in ${\Bbb Z}_p^d$. 
\end{theorem} 
\begin{proof} Again we  will use Theorem \ref{thm: planes}.  We see that the condition is equivalent to saying that $f$ is $\textit{good}$ so we can write $f$ as the sum of a constant and a linear combination of indicator functions of $d-1$-dimensional planes of the form $v \cdot x=k$  with $k \neq 0$ and $v \cdot v=0$.  Now for such a plane  and a point $x$ on it, $x+v$ is also on the plane and $(x+v) \cdot (x+v)= x \cdot x + 2k$.  This means that each such plane is uniformly distributed across all spheres centered around the origin since we can partition each plane into lines with one point on each sphere.  Now, each sphere of nonzero radius has the same number of points (counting with Jacobi sums shows that each sphere of nonzero radius has $p^d-p^{\frac{d}{2}-1}$ points) so the constant function is also equi-distributed across all spheres of nonzero radius.  To complete the proof, it suffices to note that the choice of the origin was arbitrary so we can translate it to an arbitrary point.
\end{proof}

\section{Basis of eigenfunctions of the Fourier transform}

The unnormalized prime field Fourier transform can be viewed as a linear transformation from $W \to W$ where $W$ is the $\mathbb{C}$-vector space of complex valued functions on $\mathbb{Z}_p^d$. Since 
$$ \hat{f}(m) = p^{-d} \sum_{x \in \mathbb{Z}_p^d} f(x) \chi(-x \cdot m) $$
with respect to the basis of delta-functions, it is given by a $p^d \times p^d$ Hadamard matrix $H$ of Butson type whose ``$(x,m)$-entry'' is $\chi(-x \cdot m)$ (see for example \cite{REUPAPER} for details). 

The distinct rows (and columns) of $H$ are orthogonal under the usual Hermitian inner product on $\mathbb{C}^{p^d}$. In fact $H = p^{d/2} T$ where $T$ is a unitary matrix.
Thus the eigenvalues of the unnormalized Fourier transform are complex numbers of norm $p^{d/2}$ and $H$ is unitarily diagonalizable, i.e. there is a basis of $W$ 
consisting of eigenfunctions of the (unnormalized) Fourier transform. From this it follows that there is a basis of $W$ using eigenfunctions of the regular (normalized) Fourier 
transform whose eigenvalues are complex numbers of norm $\frac{p^{d/2}}{p^d}=\frac{1}{p^{d/2}}$.

This basis of eigenfunctions of the Fourier transform forms yet another basis for the space of complex valued functions on $\mathbb{Z}_p^d$ that is useful on occasions just as the 
wavelet basis is.

Our next result characterizes which indicator functions of sets are eigenfunctions of the Fourier transform. 
In other words, it studies sets $E \subset {\Bbb Z}_p^d$ such that $\widehat{E}(m)=\lambda E(m)$ for some constant $\lambda \in {\mathbb C}$. We shall refer to such sets as self Fourier dual. 

\begin{definition} We say that a subspace $L$ of ${\Bbb Z}_p^d$ is Lagrangian if $L=L^{\perp}$. \end{definition} 

Note that Lagrangian subspaces only exist when $d$ is even and have dimension $\frac{d}{2}$. 

\begin{theorem} Let $E \subset {\Bbb Z}_p^d$ be a self Fourier dual set with $\hat{E}=\lambda E$. 
Then either $E$ is the empty set or $d$ is even, $E$ is a Lagrangian subspace and $\lambda = \sqrt{\frac{1}{p^{d}}}$. \end{theorem} 
\begin{proof}
Let $E \subseteq \mathbb{Z}_p^d$ and $\hat{E} = \lambda E$ for some $\lambda \in \mathbb{C}$. If $\lambda = 0$ then $\hat{E}=0$ and so $E = \emptyset$. Thus we are done in this case so 
assume $\lambda \neq 0$ and $E \neq \emptyset$. In this case the support of $\hat{E}$ is exactly the set $E$ and so $0 \in E$ as $\hat{E}(0) = \frac{|E|}{p^{d}} \neq 0$.

Since $\hat{E}/\lambda = E$ and $E^2=E$ we have $(\hat{E}/\lambda) \cdot (\hat{E}/\lambda) = \hat{E}/\lambda$. Taking the inverse Fourier transform of both sides and using that product becomes convolution we have $E \star E = \mu \lambda E$ for some nonzero normalization $\mu$. Given $x, y \in E$, $x+y$ is in the support of $E \star E$ and hence is in $E$ as $E \star E = \mu \lambda E$. Thus $E$ is a subspace of $\mathbb{Z}_p^d$. 

Finally when $E$ is a subspace, the support of $\hat{E}$ is exactly $E^{\perp}$, the perpendicular subspace. Thus $\hat{E} = \lambda E$ forces $E=E^{\perp}$ in this case, i.e., $E$ is a Lagrangian subspace of $\mathbb{Z}_p^d$. As $\dim(E) + \dim(E^{\perp}) = d$ in general, Lagrangian subspaces only exist when 
$d$ is even.  $\hat{E}=\lambda E$ evaluated at $0$ yields $\frac{|E|}{p^d} = \lambda$. The Theorem follows.
\end{proof}

More generally we can find eigenfunctions consisting of linear combinations of subspaces and their perpendicular subspace:

\begin{proposition}
Let $V \subseteq \mathbb{Z}_p^d$ be vector subspace of dimension $k$ and let $f_+=p^{\frac{d}{2}-k}1_V + 1_{V^{\perp}}$ 
and $f_-=p^{\frac{d}{2}-k}1_V - 1_{V^{\perp}}$. Then as long as $V$ is not a Lagrangian subspace, $f_+, f_-$ are (linearly independent) real-valued eigenfunctions of the Fourier transform corresponding 
to eigenvalues $\pm \frac{1}{p^{d/2}}$.
\end{proposition}
\begin{proof}

Note that if $V$ is a Lagrangian subspace then $f_{-}=0$. Otherwise $f_{+}$ and $f_{-}$ are nonzero functions.

A direct calculation shows that $\widehat{1_V} = \frac{|V|}{p^d} 1_{V^{\perp}} = \frac{1}{p^{d-k}} 1_{V^{\perp}}$ for any subspace $V$ of dimension $k$. 
Thus we also have $\widehat{1_{V^\perp}} = \frac{1}{p^k} 1_V$.  Then we compute:
$$ \widehat{f_+} = p^{\frac{d}{2}-k} \left(\frac{1}{p^{d-k}}1_{V^{\perp}}\right) + \frac{1}{p^k} 1_V = \frac{1}{p^{d/2}} f_{+}.$$

A similar computation works for $f_{-}$.
\end{proof}

To handle affine subspaces that do not go through the origin, we define the phase function $\phi_x(m)=\chi(-x \cdot m)$ for all $x, m \in \mathbb{Z}_p^d$. 
Note that $\bar{\phi}_x = \phi_{-x}$.

\begin{proposition}
\label{pro:eigenbasis}
Let $V \subseteq \mathbb{Z}_p^d$ be a vector subspace of dimension $k$ and let $x \in \mathbb{Z}_p^d$. 
Define $f_+=p^{\frac{d}{2}-k}1_{V+x} + \phi_{-x}1_{V^{\perp}}$ 
and $f_-=p^{\frac{d}{2}-k}1_{V+x} - \phi_{-x} 1_{V^{\perp}}$. Then as long as $V+x$ is not a Lagrangian subspace through the origin, $f_+, f_-$ are (linearly independent) eigenfunctions of the conjugate Fourier transform (Fourier transform followed by complex conjugation) corresponding 
to real eigenvalues $\pm \frac{1}{p^{d/2}}$. 
\end{proposition}
\begin{proof}
Note that if $V+x$ is a Lagrangian subspace through the origin then $V+x=V$ and $f_{-}=0$. Otherwise $f_{+}$ and $f_{-}$ are nonzero functions.

A direct calculation shows that $\widehat{1_{V+x}} = \frac{|V|}{p^d} \phi_x 1_{V^{\perp}} = \frac{1}{p^{d-k}} \phi_{x} 1_{V^{\perp}}$ for any vector subspace $V$ of dimension $k$. 
By inverse Fourier transform, $\widehat{ \phi_{-x}1_{V^{\perp}}} = \frac{1}{p^k} 1_{V+x}$. 
Then we compute:
$$ \widehat{f_+} = p^{\frac{d}{2}-k} \left(\frac{1}{p^{d-k}} \phi_x 1_{V^{\perp}}\right) + \frac{1}{p^k} 1_{V+x} = \frac{1}{p^{d/2}} \bar{f}_{+}. $$
A similar computation works for $f_{-}$.
\end{proof}

Note using the last proposition, it follows that the indicator function of any affine subspace $1_{V+x}$ can be written as a linear combination of eigenfunctions of the conjugate Fourier transform as described in that proposition. Applying this to the wavelet decomposition of any function, we see that any function is a linear combination of eigenfunctions of the conjugate Fourier transform of the form described in Proposition~\ref{pro:eigenbasis} applied to hyperplanes. 

\vskip.125in 

\section{Wavelets over ${\Bbb Z}_{p^l}$: multi-scale analysis} 

\vskip.125in 

In order to set up wavelets in ${\Bbb Z}_{p^l}^d$, we need to take a brief aside to discuss the basic geometry in this setting. Scalars in ${\Bbb Z}_{p^l}$ have a $p$-adic valuation given as follows. Let $n \in {\Bbb Z}_{p^l}$, a non-zero element. Then we can write $n=p^ju$, where $u$ is relatively prime to $p$ and hence is an invertible element (unit) in ${\Bbb Z}_{p^l}$. Recall that there are $p^l-p^{l-1}$ such units. The $p$-adic valuation of $n$, denoted by $\nu_p(n)$ is the integer $j$ in the decomposition of $n$. In particular, the $p$-adic valuation of units is $0$ and the range of $\nu_p$ is $\{0,1,2,\dots, l-1\}$. The $p$-adic norm (size) of $n$, denoted by ${||n||}_p$ is $\frac{1}{p^{\nu_p(n)}}$. Note that the higher the valuation, the smaller the norm and hence both identify the "scale" at which the element resides. 

We work in ${\Bbb Z}_{p^l}^d$, the $d$-dimensional {\it free} module over ${\Bbb Z}_{p^l}$. In this setting, the $p$-adic norm of a vector 
$v \in {\Bbb Z}_{p^l}^d$, denoted by 
$$\nu_p(v)=\min_{1 \leq i \leq d} \nu_p(v_i) \ \text{and} \ {||v||}_p=\max_{1 \leq i \leq d} {||v_i||}_p,$$ where $v_i$ is the $i$th coordinate of $v$. 

We now define a line generated by a non-zero vector $v \in {\Bbb Z}_{p^l}^d$. These lines will have different sizes depending on ${||v||}|_p$ and will be viewed as lines at different scales. 

\begin{definition} [LINES] Given a non-zero $v \in {\Bbb Z}_{p^l}^d$, define 
$$ l_v=\{av: a \in {\Bbb Z}_{p^l} \},$$ the line generated by $v$. If $\nu_p(v)=k$, then $l_v$ is isomorphic as an additive group to ${\Bbb Z}_{p^{l-k}}$ and in particular has $p^{l-k}$ points. We call such a line a {\it level} $l-k$ line. More generally we use affine lines, which are just translates of the lines defined above. \end{definition} 

Note that the valuation of a vector gets larger, the number of points on the line it generates becomes smaller. In other words, the lines at higher levels are bigger. Also observe that a general affine line at level $s$ is a union of $p$ disjoint affine lines of level $s-1$. 

\begin{definition} [HYPERPLANES] Given a non-zero $v \in {\Bbb Z}_{p^l}^d$, define the hyperplane 
$$H_v=\{x \in {\Bbb Z}_{p^l}^d: x \cdot v=0 \},$$ the hyperplane through the origin perpendicular to $v$. \end{definition} 

Observe that 
$$ |H_v|=p^{l(d-1)+\nu_p(v)}.$$

Hence as the valuation of $\nu_p$ increases, the size of the hyperplanes increases. Equivalently, as the level $v$ increases, the size of the hyperplane decreases. One way to think about the difference with the case of the line above is that lines and hyperplanes are dual. So if $v$ determines a small line, it determines a large hyperplane. 

We now define wavelets in this context. We give a definition at level $l$ (corresponding to the least degenerate lines) with the other levels defined analogously. 

\begin{definition} A wavelet at level $l$ is a function from ${\Bbb Z}_{p^l}^d \to {\Bbb C}$ whose Fourier transform is supported at an affine line at level $l$. \end{definition} 

One can establish the following equivalence using the arguments similar to those used in the field case. 
\begin{theorem} The function $f: {\Bbb Z}_{p^l}^d \to {\Bbb C}$ is wavelet of level $l$ if 
$$ f(x)=\sum c_i 1_{H_i}(x),$$ where $\{H_i\}$ is a family of parallel hyperplanes at level $l$. Here the level of the hyperplane is determined by the level of the vector it is perpendicular to. 
\end{theorem} 

We also have the following decomposition theorem. Using the principle of inclusion-exclusion and the arguments used in the field case, we have the following decomposition of functions into wavelets in this context. 

\begin{theorem} Let $f: {\Bbb Z}_{p^l}^d \to {\Bbb C}$. Then there exists a finite family of wavelets $f_i$ such that $f=\sum_i f_i$. \end{theorem} 

Once again the number of wavelets depends on the support of $\widehat{f}$, or more precisely, the minimal number of lines of level $l$ needed to cover the support of $\widehat{f}$. This naturally leads one to define the level $l$ bandwidth of $f$ in this way. We shall develop this theory further in the sequel.

\vskip.125in 

\end{document}